\documentclass[11pt,reqno]{amsart}
\usepackage{amsmath,amssymb,amsfonts}
\usepackage{mathrsfs}
\usepackage[title]{appendix}
\usepackage{mathabx}
\usepackage{color}
\definecolor{BLUE}{rgb}{0,0,1}
\definecolor{BLACK}{rgb}{0,0,0}
\definecolor{black}{rgb}{0,0,0}

\usepackage[hidelinks]{hyperref}
\usepackage{bm}

\allowdisplaybreaks
\usepackage[numbers,sort&compress]{natbib}

\makeatletter
\def\tank#1{\protected@xdef\@thanks{\@thanks
		\protect\footnotetext[0]{#1}}}
\def\bigfoot{
	
	\@footnotetext}
\makeatother

\newcommand{\ea}{\end{array}}

\allowdisplaybreaks
\numberwithin{equation}{section}

\newtheorem{theorem}{Theorem}[section]
\newtheorem{lemma}{Lemma}[section]
\newtheorem{proposition}[theorem]{Proposition}

\newtheorem{condition}[theorem]{Condition}

\newtheorem{remark}{Remark}[section]

\def\beq{\begin{equation}}
\def\nneq{\end{equation}}

\def\bthm{\begin{theorem}}
\def\nthm{\end{theorem}}

\def\blem{\begin{lemma}}
\def\nlem{\end{lemma}}
\def\bprf{\begin{proof}}
\def\nprf{\end{proof}}
\def\bprop{\begin{prop}}
\def\nprop{\end{prop}}
\def\brmk{\begin{rem}}
\def\nrmk{\end{rem}}

\def\bexa{\begin{exa}}
\def\nexa{\end{exa}}
\def\bcor{\begin{cor}}
\def\ncor{\end{cor}}

\def\cF{\mathcal{F}}

\def\cH{\mathcal{H}}

\def\RR{\mathbb{R}}

\def\EE{\mathbb{E}}

\def\pp{\mathbb{P}}

\newcommand\bp{\mathbb{P}}

\def\ee{{\mathbb E}}
\def\e{{\varepsilon}}

\newcommand{\ud}{ \mathrm{d}}

\newcommand{\FoxH}[5]{H_{#2}^{#1}\left(#3\:\middle\vert\: \begin{subarray}{l}#4\\[0.4em] #5\end{subarray}\right)}

\title[Khinchin's and Chung's LILs at time zero for an SFDE]{Khinchin's and Chung's Laws of the Iterated Logarithm at Time Zero for the Linear Stochastic Fractional Diffusion Equation}

\author[C. Liu]{Chang Liu}
\address[]{Chang Liu, School of Mathematics and Statistics,  Wuhan University,  Wuhan 430072,
China.}
\email{changliu0504@163.com}

\author[R. Wang]{Ran Wang}
\address[]{Ran Wang, School of Mathematics and Statistics,  Wuhan University,  Wuhan 430072,
China.}
\email{rwang@whu.edu.cn}

\date{}

\begin{document}
\maketitle 
\noindent {\bf Abstract.}
We consider the linear stochastic fractional diffusion equation
\begin{equation*}
\partial^{\beta} u(t,x)
=
-\left(-\Delta\right)^{\alpha/2}u(t,x)
+
I_t^{\gamma}\bigl[\dot W(t,x)\bigr],
\qquad
t>0,\quad x\in\mathbb R^d,
\end{equation*}
with zero initial conditions, where $\alpha>0$, $\beta\in(0,2)$, and $\gamma\ge0$. The driving noise $\dot W$ is a centered Gaussian generalized field that is fractional in time and has Riesz-type spatial covariance. For each fixed $x\in\mathbb R^d$, we establish a Khinchin-type law of the iterated logarithm at time zero for the temporal process $t\mapsto u(t,x)$. Under the additional conditions $0\le\gamma<1$ and $\beta+\gamma<2+H$, we also prove the corresponding Chung-type law. The proofs rely on a harmonizable representation, sharp frequency-truncation estimates, an exact small-ball asymptotic, and a localization argument. These results extend the initial-time laws of the iterated logarithm for stochastic heat equations to a broad class of time-fractional stochastic diffusion equations.

\vskip0.3cm



\section{Introduction}

Consider the following stochastic fractional diffusion equation (SFDE):
\begin{equation}\label{e:fde}
	\begin{cases}
		\partial^{\beta} u(t, x)
		=
		-\left(-\Delta\right)^{\alpha / 2} u(t, x)
		+
		I_{t}^{\gamma}\big[\dot{W}(t, x)\big],
		& t>0,\ x \in \mathbb{R}^d,   \\
		u(0, \cdot)=0,
		& \text{if } \beta \in(0,1], \\
		u(0, \cdot)=0,\quad
		\dfrac{\partial}{\partial t} u(0, \cdot)=0,
		& \text{if } \beta \in(1,2),
	\end{cases}
\end{equation}
where $\alpha>0$, $\beta\in(0,2)$, and $\gamma\ge0$.  
\textcolor{black}{Here, the symbol} $\partial^\beta$ denotes the Caputo fractional derivative  of order $\beta$. For a sufficiently regular function $f$, it is defined by 
\begin{equation}\label{e:CFDO}
	\partial^{\beta} f(t):=
	\begin{cases}
		\displaystyle
		\frac{1}{\Gamma(n-\beta)}
		\int_{0}^{t}
		\frac{f^{(n)}(\tau)}
		{(t-\tau)^{\beta+1-n}}
		\,\ud\tau,
		& \text{if } \beta \neq n, \vspace{0.2cm} \\
		\displaystyle
		\frac{\ud^n}{\ud t^n}f(t),
		& \text{if } \beta = n,
	\end{cases}
\end{equation}
where $n=\lceil\beta\rceil$ and $\Gamma(\cdot)$ denotes the Gamma function.  \textcolor{black}{The operator $I_t^\gamma$ denotes the Riemann--Liouville fractional integral of order $\gamma$, defined by} 
\begin{align}\label{eq int}
 	(I_{t}^{\gamma}f)(t)
	:=
	\frac{1}{\Gamma(\gamma)}
	\int_{0}^{t}f(r)(t-r)^{\gamma-1}\,\ud r,
	\qquad \gamma>0,
\end{align}
with the convention that $I_t^0=\mathrm{Id}$.

To define the spatial operator, we adopt the Fourier transform convention
\[
\widehat f(\xi)=\mathcal F f(\xi)
:=
\int_{\mathbb R^d} e^{-i x\cdot \xi} f(x)\,\ud x,
\qquad \xi\in\mathbb R^d .
\]
The operator $(-\Delta)^{\alpha/2}$ is  the nonnegative
pseudo-differential operator with Fourier symbol $|\xi|^\alpha$; that is,
\[
\mathcal F\big[(-\Delta)^{\alpha/2}f\big](\xi)
=
|\xi|^\alpha \widehat f(\xi),
\qquad \xi\in\mathbb R^d,
\]
for suitable functions $f$. When $0<\alpha\le2$, this operator is 
the usual fractional Laplacian. For $\alpha>2$,  it is understood as the $\alpha/2$-th power of the nonnegative operator $-\Delta$.

\textcolor{black}{The driving noise $\dot W$ is a centered Gaussian generalized field defined on a probability space}
$(\Omega,\mathcal F,\mathbb P)$. It is fractional in time with Hurst index
$H_0\in[1/2,1)$ and has a Riesz-type spatial covariance structure. \textcolor{black}{More precisely, for suitable test functions} $\varphi$ and $\psi$, 
\begin{equation}\label{e:covariance-noise-2} 
\EE[W(\varphi)W(\psi)]:=\begin{cases}\displaystyle
	(2\pi)^{-1}\int_{\RR} \int_{\RR^d}|\xi|^{\ell-d}\mathcal F\varphi(*,\cdot)(\tau,\xi)\overline{\cF\psi(*,\cdot)(\tau,\xi)}d\tau d\xi,& H_0=1/2, \\[2.5ex]
	\displaystyle
	(2\pi)^{-1}\int_{\RR} \int_{\RR^d}|\tau|^{1-2H_0}|\xi|^{\ell-d}\mathcal F\varphi(*,\cdot)(\tau,\xi)\overline{\cF\psi(*,\cdot)(\tau,\xi)}d\tau d\xi, & H_0\in(1/2,1),
\end{cases}
\end{equation}
\textcolor{black}{where $\cF\varphi$ denotes the space--time Fourier transform of $\varphi$, namely,}
$$
\cF\varphi(*,\cdot)(\tau,\xi):=\int_{\RR}\int_{\RR^d}e^{-i(\tau t+\xi\cdot x)}\varphi(t,x)dxdt.
$$ 

Equation \eqref{e:fde} encompasses several important models, including the stochastic heat equation and its fractional variants. Fine sample path properties of solutions to these equations, such as H\"older regularity, local moduli of continuity, small-ball probabilities, and laws of the iterated logarithm, have been studied extensively; see, for example, \cite{CGS2024,CHN2019,CK2019,CLX2026,GSWX2025}. In particular, Khinchin-type and Chung-type laws of the iterated logarithm provide sharp almost sure characterizations of temporal oscillations at small time scales.

Following the pinned-string approach of Mueller and Tribe \cite{MT02}, the solution to the linear stochastic fractional heat equation with $\beta=1$ and $\gamma=0$ can be decomposed into two Gaussian random fields,  one having stationary increments and the other having smoother sample paths; see \cite{LN2009, HSWX2020, KT2019, TX17, WZ2021}. Further background and applications of this decomposition to fine sample path properties can be found in  \cite[Section 3.3]{K2014} and \cite[Section 3.2]{DS26}.  For the more general   equation \eqref{e:fde}, Guo et al. \cite{GSWX2025} derived an analogous  decomposition of $t\mapsto u(t,x)$ into a fractional Brownian motion  (FBM, for short)  and a smoother Gaussian remainder.  Consequently, at every fixed $t>0$,  the temporal process has the same local  behavior as FBM and satisfies the corresponding Khinchin-type and Chung-type LILs.

The situation at  $t=0$   is fundamentally different.  As $t\downarrow0$, the smoother remainder in the decomposition is no longer negligible relative to the FBM component
 and may contribute at the same order under the normalization relevant to the LILs. Consequently, the  initial-time   LILs for $u(t,x)$  cannot be deduced directly from the arguments used at fixed positive times.

\textcolor{black}{Recently, Lee and Xiao \cite{LX2023} developed a general framework for establishing Chung-type LILs and exact moduli of continuity for centered Gaussian random fields, based on harmonizable representations and strong local nondeterminism.} Building on this framework, Chen et al. \cite{CLX2026}  established  strong local nondeterminism in both temporal and spatial variables  for solutions to stochastic time-fractional slow and fast diffusion equations. They further obtained exact moduli of continuity, Chung-type LILs, and small-ball probability estimates. However, their results are  restricted to time intervals bounded away from zero, and the underlying arguments based on harmonizable representations and strong local nondeterminism do not apply directly near the initial time. 

\textcolor{black}{The purpose of this paper is to fill this gap by establishing Khinchin-type and Chung-type LILs at time zero for the temporal process $t\mapsto u(t,x)$, for every fixed $x\in\mathbb R^d$.}

\begin{condition}\label{hypo} Assume that 
	$$
	\alpha>0,\quad \beta\in(0,2),\quad H_0\in[1/2,1),\quad \ell\in(0,2d\wedge 2\alpha),\quad \gamma\ge 0,  
	$$
    and 
\begin{equation}\label{def:H}
H:=\beta+\gamma+H_0-1-\frac{\beta\ell}{2\alpha}>0.
\end{equation}
\end{condition}
 
By  Chen et al. \cite[Theorem 1.1(i)]{CLX2026},  under Condition \ref{hypo}, 
  equation \eqref{e:fde} admits a unique random field solution  given by
\begin{equation}\label{def:u}
u(t, x)=  \int_0^t \int_{\RR^d}G(t-s,x-y)W(ds,dy), \ \ \text{a.s.},
\end{equation}
where
\[
G(t,x) := \pi^{-d/2}|x|^{-d}t^{\beta+\gamma-1}  \FoxH{2,1}{2,3}{\frac{ |x|^\alpha}{2^{\alpha} t^\beta}}
{(1,1),\:(\beta+\gamma,\beta)}{(d/2,\alpha/2),\:(1,1),\:(1,\alpha/2)}, \quad \text{for } t>0,
\]
and  $G(t, x):=0$ for $t\le 0$, where  $H^{m,n}_{p,q}(z)$ is the Fox $H$-function. See \cite[Theorem 2.8]{CGS2024}, \cite{KS2004}, and Section \ref{sec regularity} below.

\begin{theorem}[Khinchin's LIL]\label{Khinchin:u}
	\textcolor{black}{Suppose  Condition~\ref{hypo} holds and that  $H<1$. Then, for every fixed} $x\in\RR^d$,
	\begin{equation}\label{Khinchin:u:equ}
		\begin{split}
			\limsup_{t \downarrow0} \frac{|u(t, x)|}{t^{H}\sqrt{2 \log\log(1/t)}}=\widetilde{\kappa},\ \ \ a.s.,
		\end{split}
	\end{equation}
	where $\widetilde\kappa$ is defined by
	\begin{equation}\label{def:tilde:kappa}
		\widetilde\kappa=
		\begin{cases}
			\displaystyle
            \Biggl(\int_{\mathbb{R}^d} d\xi\,|\xi|^{\ell-d}
			\int_{[0,1]} ds\, s^{2(\beta+\gamma-1)}
			E_{\beta,\beta+\gamma}^2\left(-|\xi|^\alpha s^\beta\right)
			\Biggr)^{\!1/2}, & H_0=1/2,
            \\[2.5ex]
            \Biggl(c_{H_0}\int_{\mathbb{R}^d} d\xi\,|\xi|^{\ell-d}
			\int_{0}^1 \int_0^1ds_1ds_2\,|s_1-s_2|^{2H_0-2}\\
			\qquad\cdot s_1^{\beta+\gamma-1}
			E_{\beta,\beta+\gamma}\left(-|\xi|^\alpha s_1^\beta\right)\,
			s_2^{\beta+\gamma-1}
			E_{\beta,\beta+\gamma}\left(-|\xi|^\alpha s_2^\beta\right)
			\Biggr)^{\!1/2}, & H_0\in(1/2,1).
		\end{cases}
	\end{equation}

\end{theorem}

\textcolor{black}{By \cite[Theorem~4.1(i)]{CLX2026}, the constant $\widetilde{\kappa}$ belongs to $(0,\infty)$ under Condition~\ref{hypo}.}

\begin{theorem}[Chung's LIL]\label{chung:u} 
	 Suppose that Condition \ref{hypo} holds, $H<1$,   {$0\le\gamma<1$}, and $\beta+\gamma<2+H$.
	Then, for every  fixed $x\in\RR^d$,
	\begin{equation}\label{chung:u:equ}
		\liminf_{\varepsilon\downarrow0} \sup_{t\in[0,\varepsilon]}\frac{|u(t,x)|}{\e^{H}(\log\log \e^{-1})^{-H}}
		=\kappa\lambda_H^{H},\ \ \ a.s.,
	\end{equation}
	where   
	\begin{equation}\label{def:kappa}
		\textcolor{black}{\kappa:=
		\Bigg(\frac{1}{\pi}\int_{\mathbb{R}}\frac{1-\cos\tau}{|\tau|^{2H+1}}\,d\tau
		\int_{\mathbb{R}^d}\frac{|\xi|^{\ell-d}}
		{1+2|\xi|^\alpha\cos(\pi\beta/2)+|\xi|^{2\alpha}}\,d\xi
		\Bigg)^{1/2}.}
	\end{equation}
 \textcolor{black}{Here,} $\lambda_H$ is the small-ball constant of a standard FBM $\{B_H(t)\}_{t\ge0}$ with index $H$ (see, e.g., \cite{LS2001}), defined by
 \begin{equation}\label{def:lambda:H}
\lambda_H:=-\lim_{\varepsilon\downarrow0}\varepsilon^{1/H}
\log\mathbb P\left(\sup_{0\le t\le1}|B_H(t)|\le\varepsilon\right).
\end{equation}
\end{theorem}
\begin{remark}
  {The condition $0\le\gamma<1$ is inherited from the exact small-ball theorem used in Proposition~\ref{small-ball:u}; see \cite[Theorem~4.3(1)]{GSWX2025}. The additional restriction $\beta+\gamma<2+H$ is needed to control the smoother Gaussian remainder in that result. Extending the exact small-ball asymptotic, and hence the Chung-type LIL, beyond either restriction would require a separate argument.}
\end{remark}

Several recent works are closely related to the present study. Building on the framework of Lee and Xiao \cite{LX2023}, Qian et al. \cite{QWWX2026} used local linearization and truncated harmonizable representations to establish a Khinchin-type LIL at $t=0$. For the Chung-type LIL, Khoshnevisan, Kim, and Mueller \cite{KKM24} treated the stochastic heat equation driven by space-time white noise. Their argument combines small-ball estimates for the fractional Brownian component and the smoother remainder in the pinned-string decomposition with a localization method. Liu and Wang \cite{LW2026} later extended this approach to the linear stochastic fractional heat equation driven by noise that is white in time and fractional in space, obtaining a Chung-type LIL at $t=0$.

A further motivation for studying the initial-time behavior of the linear equation comes from local linearization results for nonlinear parabolic SPDEs. A central principle is that, at sufficiently small scales, the increments of a nonlinear solution are often approximated to leading order by those of the corresponding linear equation. For instance, Foondun, Khoshnevisan, and Mahboubi \cite{FKM2015} characterized the spatial approximate gradient of a nonlinear stochastic fractional heat equation through increments of fractional Brownian motion. Hairer and Pardoux \cite{HP15} derived a sharp local expansion for nonlinear parabolic SPDEs, with a leading linearized term and a higher-order remainder.

In the temporal setting, Khoshnevisan et al. \cite{KSXZ2013} developed a local approximation for stochastic fractional heat equations with multiplicative noise: at a fixed positive space-time point, a temporal increment of the nonlinear solution is approximated by the corresponding increment of the linear equation multiplied by the local noise coefficient, with a higher-order error controlled in moments. Related local approximation and linearization results can be found in \cite{CHKK19,DNP2025,Das2022,HK2017,KKM24,QWWX2026}. The Khinchin-type and Chung-type LILs established here therefore provide sharp initial-time benchmarks for the linear equation and may serve as a starting point for studying the corresponding initial-time behavior of nonlinear stochastic heat-type equations.

The remainder of the paper is organized as follows. In Section~2, we introduce the covariance structure of the noise and the associated stochastic integral, and establish the H\"older continuity and a harmonizable-type representation of the solution. In Section~3, we prove the Khinchin-type LIL. In Section~4, we establish the required small-ball estimate and prove the Chung-type LIL.

\section{Preliminaries}

\subsection{Covariance structure and stochastic integration}
  {The spatial covariance is specified directly through the nonnegative tempered spectral measure}
\[
  {m(d\xi):=|\xi|^{\ell-d}\,d\xi,\qquad 0<\ell<2d.}
\]
  {When $0<\ell<d$, this measure is, up to the normalization determined by the Fourier convention, associated with the Riesz kernel $|x|^{-\ell}$. When $d\le\ell<2d$, its inverse Fourier transform is understood as a tempered distribution. In both cases, the stochastic integral is defined through the Hilbert-space completion induced by the covariance form below.}

We adopt the noise model introduced in \cite[p. 7]{CLX2026}. Let $(\Omega, \mathcal F, \mathbb P)$ be a complete probability space and let $\mathcal D(\RR\times\RR^d)$ denote the space of real-valued infinitely differentiable functions with compact support in $\RR\times\RR^d$. Recall \eqref{e:covariance-noise-2}. The noise $\dot{W}$ is a zero-mean Gaussian family $\{W(\varphi); \varphi\in\mathcal D(\RR\times\RR^d)\}$, whose covariance structure is also given by
\begin{equation}\label{e:covariance-noise-1} 
\EE[W(\varphi)W(\psi)]:=\begin{cases}\displaystyle
	\int_{\RR}\int_{\RR^d}\cF\varphi(t,\cdot)(\xi) \overline{\cF\psi(t,\cdot)(\xi)} |\xi|^{\ell-d}d\xi dt, & H_0=1/2, \\[2.5ex]
	\displaystyle
	c_{H_0}\int_{\RR}\int_{\RR}dsdt\int_{\RR^d} |t-s|^{2{H_0}-2}\cF\varphi(t,\cdot)(\xi) \overline{\cF\psi(s,\cdot)(\xi)} |\xi|^{\ell-d}d\xi, & H_0\in(1/2,1).
\end{cases}
\end{equation}
Here,
\[
  {c_{H_0}:=2^{2-2H_0}\pi^{1/2}\frac{\Gamma(1-H_0)}{\Gamma(H_0-1/2)}.}
\]
\textcolor{black}{The constant $c_{H_0}$ is chosen so that} the time-domain and frequency-domain covariance representations agree under our Fourier convention. More precisely, for suitable test functions $f,g$ and $H_0\in(1/2,1)$,
\[
 c_{H_0}\int_{\mathbb R^2}|t-s|^{2H_0-2}f(t)g(s)dtds
 =(2\pi)^{-1}\int_{\mathbb R}|\tau|^{1-2H_0} 
 \cF f(\cdot)(\tau)  \overline{ \cF g(\cdot)(\tau) }d\tau .
\]
See \cite[p. 3]{TX17} and \cite[p. 60]{BT08} for more information.  

When $H_0=1/2$ and $\ell=d$, the noise $\dot W$ reduces to space-time white noise.
Let $\mathcal H$ be the completion of $\mathcal D(\RR\times\RR^d)$ with the inner product
$$
\langle \varphi,\psi \rangle_\cH=\EE[W(\varphi)W(\psi)].
$$
\textcolor{black}{The isometry $\varphi\mapsto W(\varphi)$ extends uniquely to $\mathcal H$; the resulting random variable is called the Wiener integral of $\varphi$.} For each $\varphi\in\mathcal H$, we also use the notation
$$
W(\varphi)=\int_{\RR}\int_{\RR^d}\varphi(t,x)W(dt,dx).
$$

\subsection{H\"older continuity of the solution}\label{sec regularity}

If
$G(t-*,x-\cdot)\bm{1}_{[0,t]}(\cdot)\in\cH$ for all $(t,x)\in[0,\infty)\times\RR^d$, then the solution $\{u(t,x):(t,x)\in[0,\infty)\times\RR^d\}$ is a centered Gaussian process with covariance function
\begin{equation}\label{cov:u}
\ee [u(t,x)u(s,y)]=\langle G(t-*,x-\cdot)\bm{1}_{[0,t]}(*),G(s-*,y-\cdot)\bm{1}_{[0,s]}(*) \rangle_\cH.
\end{equation} 

The Fourier transform of $G(t, x)$ in space is given by
\begin{equation}\label{eq Fourier p}
\cF G(t,\cdot)(\xi)=t^{\beta+\gamma-1}E_{\beta,\beta+\gamma}\left(- t^\beta|\xi|^\alpha\right),
\end{equation}
\textcolor{black}{where $E_{\beta,\beta+\gamma}$ is the two-parameter Mittag--Leffler function} (see, e.g., \cite[Sect.1.2]{P1999}):
\begin{equation}\label{def-MLF}
E_{a, b}(z):=\sum_{k=0}^{\infty}\frac{z^k}{\Gamma(a k+b)},\,\,\text{for all}\,\, a\in\RR_+,b\in\RR,\,\,\text{and}\,\,z\in\mathcal C.
\end{equation}
The reciprocal Gamma function is understood through its entire extension; in particular, we use the convention (see, e.g., \cite[(5.2.1) on p. 136]{FDRC2010})
\[
1/\Gamma(z)\equiv0\quad\text{for}\quad z=0,-1,-2,\cdots.
\]

The Mittag--Leffler function defined in \eqref{def-MLF} satisfies the following properties:
\begin{itemize}
\item[(1)]\cite[Theorem 1.6]{P1999}
  {For every fixed $a\in(0,2)$ and $b\in\mathbb R$, there exists a constant $C_{a,b}>0$ such that, for all $x>0$,}
\begin{equation}\label{Upper:E}
  {|E_{a,b}(-x)|\leq \frac{C_{a,b}}{1+x}.}
\end{equation}
  {All Mittag--Leffler parameters used below are fixed, so the corresponding constants are absorbed into the generic constants $c_{i,j}$.}
\item[(2)]
\cite[(2.5)]{CLX2026} For any $a>0, b,\lambda\in\RR$, and $m=1,2,3,\cdots$, \textcolor{black}{we have}
\begin{equation}\label{par}
	\frac{d^m}{dx^m}\left(x^{b-1}E_{a,b}(\lambda x^a)\right)=x^{b-1-m}E_{a,b-m}(\lambda x^a).
\end{equation}
\end{itemize}

\begin{lemma}\label{solution:pro}
Recall that $H$ is defined in \eqref{def:H}. Suppose that  Condition \ref{hypo}  holds  and that $H<1$.  Then, for every fixed $x\in\mathbb{R}^d$, the following assertions hold.
\begin{itemize}
	\item[(1)] For every $t\ge0$,
	\begin{equation}\label{u-L2}
		\|u(t,x)\|_{L^2(\Omega)}^2
		=
		\textcolor{black}{\widetilde{\kappa}}^2 t^{2H},
	\end{equation}
	where $\textcolor{black}{\widetilde{\kappa}}$ is given by \eqref{def:tilde:kappa}.
	
	\item[(2)] The process $\{u(t,x)\}_{t\ge0}$ is a centered Gaussian self-similar process with Hurst index $H$.
\end{itemize}

\end{lemma}

\begin{proof}
 \textcolor{black}{Assertion (1) follows from \cite[Theorem 4.1(i)]{CLX2026}.} 
We now prove assertion (2). Since $\{u(t,x)\}_{t\ge0}$ is a centered Gaussian
process, it is enough to verify the corresponding scaling property of its
covariance function. We give the proof for $H_0\in(1/2,1)$; the case
$H_0=1/2$ is analogous.

Let $\rho>0$ and $t,s\ge0$. By \eqref{e:covariance-noise-1} and
\eqref{eq Fourier p}, we have
\[
\begin{aligned}
&\mathbb E\left[
	\rho^{-H}u(\rho t,x)\,
	\rho^{-H}u(\rho s,x)
\right]  \\
&=
c_{H_0}\rho^{-2H}
\int_0^{\rho t}\!\!\int_0^{\rho s}
|r_1-r_2|^{2H_0-2}
\int_{\mathbb R^d}
\mathcal F G(\rho t-r_1,x-\cdot)(\xi)\,
\overline{\mathcal F G(\rho s-r_2,x-\cdot)(\xi)}
|\xi|^{\ell-d}\,\ud\xi\,\ud r_2\,\ud r_1 .
\end{aligned}
\]
The phase factors depending on $x$ cancel in the product of the two Fourier
transforms. Therefore, using \eqref{eq Fourier p}, the last display becomes
\[
\begin{aligned}
&c_{H_0}\rho^{-2H}
\int_0^{\rho t}\!\!\int_0^{\rho s}
|r_1-r_2|^{2H_0-2}
\int_{\mathbb R^d}
(\rho t-r_1)^{\beta+\gamma-1}
E_{\beta,\beta+\gamma}
\left(-(\rho t-r_1)^\beta|\xi|^\alpha\right)    \\
&\qquad\qquad\qquad\qquad \times
(\rho s-r_2)^{\beta+\gamma-1}
E_{\beta,\beta+\gamma}
\left(-(\rho s-r_2)^\beta|\xi|^\alpha\right)
|\xi|^{\ell-d}\,\ud\xi\,\ud r_2\,\ud r_1 .
\end{aligned}
\]
Set
\[
r_i=\rho \widetilde r_i,\qquad i=1,2,
\qquad
\widetilde\xi=\rho^{\beta/\alpha}\xi .
\]
Then
\[
\ud r_1\,\ud r_2=\rho^2\,\ud\widetilde r_1\,\ud\widetilde r_2,
\qquad
|r_1-r_2|^{2H_0-2}
=
\rho^{2H_0-2}
|\widetilde r_1-\widetilde r_2|^{2H_0-2},
\]
and
\[
|\xi|^{\ell-d}\,\ud\xi
=
\rho^{-\beta\ell/\alpha}
|\widetilde\xi|^{\ell-d}\,\ud\widetilde\xi .
\]
Using the definition
\[
H=\beta+\gamma+H_0-1-\frac{\beta\ell}{2\alpha},
\]
all powers of $\rho$ cancel. Hence
\[
\begin{aligned}
&\mathbb E\left[
	\rho^{-H}u(\rho t,x)\,
	\rho^{-H}u(\rho s,x)
\right] \\
&=
c_{H_0}
\int_0^t\!\!\int_0^s
|\widetilde r_1-\widetilde r_2|^{2H_0-2}
\int_{\mathbb R^d}
(t-\widetilde r_1)^{\beta+\gamma-1}
E_{\beta,\beta+\gamma}
\left(-(t-\widetilde r_1)^\beta|\widetilde\xi|^\alpha\right) \\
&\qquad\qquad\qquad\qquad \times
(s-\widetilde r_2)^{\beta+\gamma-1}
E_{\beta,\beta+\gamma}
\left(-(s-\widetilde r_2)^\beta|\widetilde\xi|^\alpha\right)
|\widetilde\xi|^{\ell-d}\,\ud\widetilde\xi\,
\ud\widetilde r_2\,\ud\widetilde r_1  \\
&=
\mathbb E\left[u(t,x)u(s,x)\right].
\end{aligned}
\]

When $H_0=1/2$, the double time integral is replaced by a single time integral; the same substitutions $r=\rho\tilde r$ and $\tilde\xi=\rho^{\beta/\alpha}\xi$ give the factor $\rho^{2H}$.

Thus, for every $\rho>0$, the centered Gaussian processes
\[
\left\{\rho^{-H}u(\rho t,x):t\ge0\right\}
\quad\text{and}\quad
\left\{u(t,x):t\ge0\right\}
\]
have the same covariance function, and hence the same finite-dimensional
distributions.   
The proof is complete.
\end{proof}

\textcolor{black}{By \cite[Proposition 5.1(i)]{CLX2026}}, for any $0<S<T$,  there exists a positive constant $C_{S,T}$,  depending only  on $S$ and $T$,  such that for all $t,s\in[S, T]$ and  $x\in\mathbb{R}^d$,
\begin{equation}\label{CLX-Upper}
\|u(t,x)-u(s,x)\|_{L^2(\Omega)}\le C_{S,T}|t-s|^H.
\end{equation}
\textcolor{black}{Using the self-similarity} of the process $\{u(t, x)\}_{t\ge0}$, we can further obtain the following H\"older estimate,  where the constant is independent of the lower and upper bounds of the time interval.

\begin{proposition}\label{solution:holder}
Recall that $H$ is defined in \eqref{def:H}. Suppose that  Condition \ref{hypo}  holds  and that $H<1$. Then there exists a constant
$c_{2,1}:=\max\left\{C_{1/2,1}, 2^{H+1}\textcolor{black}{\widetilde{\kappa}}\right\}>0$ such that, for all $t\ge s\ge0$ and $x\in\mathbb{R}^d$,
\begin{equation}\label{equ-metric-d1}
	\|u(t,x)-u(s,x)\|_{L^2(\Omega)}
	\le c_{2,1} |t-s|^H.
\end{equation}
\end{proposition}

\begin{proof}
Fix $x\in\mathbb{R}^d$. By Lemma \ref{solution:pro}, for all $a> b\ge0$ and all $\rho>0$,
\begin{equation}\label{self}
	\mathbb{E}\left[|u(\rho a,x)-u(\rho b,x)|^2\right]
	=
	\rho^{2H}
	\mathbb{E}\left[|u(a,x)-u(b,x)|^2\right].
\end{equation} 
We distinguish two cases.

First, suppose that $s\ge t/2$. Then $s/t\in[1/2,1]$. Taking $\rho=t$, $a=1$, and $b=s/t$ in \eqref{self}, and then using \eqref{CLX-Upper} on the interval $[1/2,1]$, we obtain
\begin{align*}
	\mathbb{E}\left[|u(t,x)-u(s,x)|^2\right]
	= &\, 
	t^{2H}
	\mathbb{E}\left[|u(1,x)-u(s/t,x)|^2\right]   \\
	\le &\, 
	C_{1/2,1}^2 t^{2H} |1-s/t|^{2H}  
	=
	C_{1/2,1}^2 |t-s|^{2H}.
\end{align*}

Next, suppose that $0\le s<t/2$. Then $t-s>t/2$. By \eqref{u-L2} and the elementary inequality $$|a-b|^2\le 2|a|^2+2|b|^2,$$ we have
\begin{align*}
	\mathbb{E}\left[|u(t,x)-u(s,x)|^2\right]
	\le &\, 
	2\mathbb{E}\left[|u(t,x)|^2\right]
	+
	2\mathbb{E}\left[|u(s,x)|^2\right]  \\
	\le &\, 
	2\textcolor{black}{\widetilde{\kappa}}^2 t^{2H}
	+
	2\textcolor{black}{\widetilde{\kappa}}^2 s^{2H} 
	\le
	2^{2H+2}\textcolor{black}{\widetilde{\kappa}}^2 |t-s|^{2H}.
\end{align*}

\textcolor{black}{Taking square roots in the two cases and recalling the definition of $c_{2,1}$, we obtain}
\[
\textcolor{black}{\|u(t,x)-u(s,x)\|_{L^2(\Omega)}\le c_{2,1}|t-s|^H,}
\]
\textcolor{black}{which is exactly \eqref{equ-metric-d1}.} The proof is complete.
\end{proof}

\textcolor{black}{Since $u(\cdot,x)$ is Gaussian, \eqref{equ-metric-d1} implies, for every $p\ge2$,}
\[
\textcolor{black}{\mathbb E|u(t,x)-u(s,x)|^p\le C_p|t-s|^{pH}.}
\]
\textcolor{black}{Choosing $p>1/H$ and applying Kolmogorov's continuity theorem, we obtain a continuous modification of $u(\cdot,x)$ on every compact interval. Throughout the sequel, we work with this continuous modification.}

\subsection{Harmonizable representation for the solution}
Let $G_{t,x}(s,y)=G(t-s,x-y)\bm{1}_{[0,t]}(s)$. \textcolor{black}{A direct computation shows that} the Fourier transform of $G_{t,x}(*,\cdot)$ is 
\begin{equation}\label{Four:G:tx}
\mathcal FG_{t,x}(\tau,\xi)=e^{-ix\cdot\xi}e^{-i\tau t}\int_0^te^{i\tau s}s^{\beta+\gamma-1}E_{\beta, \beta+\gamma}\left(-s^\beta|\xi|^\alpha\right)ds,
\end{equation}
see \cite[(6.8)]{CLX2026}.
Let
$$
\widetilde{W}=\widetilde{W}_1+i\widetilde{W}_2,
$$
where $\widetilde{W}_1$ and $\widetilde{W}_2$ are two independent space-time Gaussian white noises on $\mathbb{R}\times\mathbb{R}^d$. For each $(t,x)\in\mathbb{R}_+\times\mathbb{R}^d$ and $A\in\mathcal{B}(\mathbb{R}_+)$, by \eqref{e:covariance-noise-2}, define
\begin{equation}\label{def:vx}
v(A,t,x)
:=
\frac{1}{\sqrt{2\pi}}
\operatorname{Re}
\iint_{\{|\tau|^H\in A,\,\xi\in\mathbb{R}^d\}}
\mathcal{F}G_{t,x}(\tau,\xi)
|\tau|^{\frac{1-2H_0}{2}}
|\xi|^{\frac{\ell-d}{2}}
\widetilde{W}(d\tau,d\xi).
\end{equation}
Then the centered Gaussian random field
$$
\left\{v(\mathbb{R}_+,t,x):t\ge0,\ x\in\mathbb{R}^d\right\}
$$
has the same finite-dimensional distributions as the solution
$$
\left\{u(t,x):t\ge0,\ x\in\mathbb{R}^d\right\}
$$
to \eqref{e:fde}. Indeed, by \eqref{e:covariance-noise-1} and \eqref{cov:u}, the two centered Gaussian fields have the same covariance function.
\textcolor{black}{For each fixed $(t,x)\in\mathbb{R}_+\times\mathbb{R}^d$, the map}
\[
\textcolor{black}{A\longmapsto v(A,t,x),\qquad A\in\mathcal B(\mathbb R_+),}
\]
\textcolor{black}{is an independently scattered Gaussian random measure. More generally, Gaussian fields obtained by restricting the integral in \eqref{def:vx} to disjoint Borel subsets of the frequency variable $|\tau|^H$ are independent. In particular, the processes used in the block decompositions below are independent whenever their frequency blocks are disjoint; see \cite[Section 8]{CLX2026}.}

\begin{lemma}\label{U:JR}
Recall that $H$ is defined in \eqref{def:H}. Suppose that Condition \ref{hypo}  holds  and that $H<1$.
For $R>0$, define
\begin{equation}\label{def:J}
	J(R):=\int_{\RR^d}\left| \int_0^{R}e^{ir}r^{\beta+\gamma-1}E_{\beta, \beta+\gamma}\left(-r^\beta|\eta|^\alpha\right)dr\right|^2 |\eta|^{\ell-d} d\eta.
\end{equation}
Then there exists a constant $c_{2,2}>0$, independent of $R$, such that
\begin{equation}\label{U:JR:eq}
	J(R)\le c_{2,2}\begin{cases}
		R^{2(H+1-H_0)},  & 0<R\le1, \\[2.5ex]
		1, &  R>1, H<H_0, \\[2.5ex]
		(\log R)^2,&  R>1, H=H_0, \\[2.5ex]
		R^{2(H-H_0)},& R>1, H>H_0.
	\end{cases}
\end{equation}
\end{lemma}

\begin{proof} 
Using polar coordinates \(\eta=\rho w\), \(\rho>0\), \(w\in\mathbb S^{d-1}\),
and the identity
\[
        |\mathbb S^{d-1}|=\frac{2\pi^{d/2}}{\Gamma(d/2)},
\]
we obtain
\begin{equation}\label{def:J:eq}
        J(R)
        =
        \frac{2\pi^{d/2}}{\Gamma(d/2)}
        \int_0^\infty
        \left|
        \int_0^{R}
        e^{ir}r^{\beta+\gamma-1}
        E_{\beta,\beta+\gamma}\bigl(-r^\beta\rho^\alpha\bigr)
        \,dr
        \right|^2
        \rho^{\ell-1}\,d\rho .
\end{equation}
\textbf{Step 1.} The case $0<R\le 1$. 
By \eqref{Upper:E} and the change of variables  $u:=R^{-1}r$ and $\zeta:=R^{\frac\beta\alpha}\rho$, we have
\begin{equation*}
	\begin{split}
		J(R)\le&\,  c_{2,3}\int_0^\infty \left|\int_0^R \frac{r^{\beta+\gamma-1}}{1+r^\beta\rho^\alpha}dr\right|^2 \rho^{\ell-1} d\rho\\
		=&\, c_{2,3}R^{2(\beta+\gamma)-\frac{\beta \ell}\alpha}\int_0^\infty \left|\int_0^1 \frac{u^{\beta+\gamma-1}}{1+u^\beta\zeta^\alpha}du\right|^2 \zeta^{\ell-1} d\zeta\\
	\end{split}
\end{equation*}
\textcolor{black}{We claim that}
\begin{equation}\label{integral:1}
	\int_0^\infty\left|\int_0^1 \frac{u^{\beta+\gamma-1}}{1+u^\beta\zeta^\alpha}du\right|^2 \zeta^{\ell-1} d\zeta<\infty.
\end{equation}
Since 
\[
2(\beta+\gamma)-\frac{\beta \ell}\alpha=2(H+1-H_0),
\]
we conclude that there exists a constant $c_{2,6}>0$, independent of $R$, such that for all $0<R\le 1$,
\begin{equation}\label{JR1}
	J(R)\le c_{2,6}R^{2(H+1-H_0)}.
\end{equation}

It remains to verify the claim \eqref{integral:1}.
 We split the integral into two parts:
\begin{equation}\label{integral:1:two}
	\int_0^\infty\left|\int_0^1 \frac{u^{\beta+\gamma-1}}{1+u^\beta\zeta^\alpha}du\right|^2 \zeta^{\ell-1} d\zeta
	=\left(\int_0^1+ \int_1^\infty  \right) \left|\int_0^1 \frac{u^{\beta+\gamma-1}}{1+u^\beta\zeta^\alpha}du\right|^2 \zeta^{\ell-1} d\zeta.
\end{equation}
Since $\beta+\gamma>0$ and $\ell>0$, we have
\begin{equation*}
	\int_0^1\left|\int_0^1 \frac{u^{\beta+\gamma-1}}{1+u^\beta\zeta^\alpha}du\right|^2 \zeta^{\ell-1} d\zeta
	\le \int_0^1\left|\int_0^1 u^{\beta+\gamma-1}du\right|^2 \zeta^{\ell-1} d\zeta= \frac1{\ell(\beta+\gamma)^2}.
\end{equation*}
For the second term in \eqref{integral:1:two}, we consider the  two cases:
\[
\gamma>0\quad\text{and}\quad\gamma=0.
\]
\textbf{(i)} When $\gamma>0$, we have
\begin{align*}
	\int_1^\infty\left|\int_0^1 \frac{u^{\beta+\gamma-1}}{1+u^\beta\zeta^\alpha}du\right|^2 \zeta^{\ell-1} d\zeta
	\le &\, \int_1^\infty\left|\int_0^1 u^{\gamma-1} du\right|^2  \zeta^{\ell-1-2\alpha} d\zeta\\
	=&\,\frac1{\gamma^2}\int_1^\infty\zeta^{\ell-1-2\alpha} d\zeta <\infty,
\end{align*}
\textbf{(ii)} When $\gamma=0$, 
using the change of variables $t:=\log\zeta$ yields
\begin{equation*}
	\begin{split}
		\int_1^\infty\left|\int_0^1 \frac{u^{\beta-1}}{1+u^\beta\zeta^\alpha}du\right|^2 \zeta^{\ell-1} d\zeta
		=&\, \frac1{\beta^2}\int_1^\infty\bigl(\log(1+\zeta^\alpha)\bigr)^2\zeta^{\ell-1-2\alpha}d\zeta\\
		\leq&\,c_{2,5}\int_1^\infty(\log\zeta)^2\zeta^{\ell-1-2\alpha}d\zeta\\
		=&\,c_{2,5}\int_0^\infty t^2e^{-t(2\alpha-\ell)}dt<\infty.
	\end{split}
\end{equation*}
 \textbf{Step 2.} We now consider  $R>1$. By \eqref{def:J:eq} and $|z_1+z_2|^2\le2|z_1|^2+2|z_2|^2$, we have
\begin{equation}\label{J:two}
	\textcolor{black}{J(R)\leq \frac{4\pi^{d/2}}{\Gamma(d/2)}\left(\int_0^\infty|J_1(\rho)|^2\rho^{\ell-1}d\rho
	+\int_0^\infty|J_2(\rho)|^2\rho^{\ell-1}d\rho\right),}
\end{equation}
where
\begin{equation*}
	J_1(\rho):=\int_0^1e^{ir}r^{\beta+\gamma-1}E_{\beta, \beta+\gamma}\left(-r^\beta\rho^\alpha\right)dr,
\end{equation*}
\begin{equation*}
	J_2(\rho):=\int_1^Re^{ir}r^{\beta+\gamma-1}E_{\beta, \beta+\gamma}\left(-r^\beta\rho^\alpha\right)dr.
\end{equation*}

For the first term in \eqref{J:two}, by  \eqref{Upper:E} and \eqref{integral:1}, we have
\begin{equation}\label{J1}
	\int_0^\infty|J_1(\rho)|^2\rho^{\ell-1}d\rho\le c_{2,6}\int_0^\infty\left|\int_0^1 \frac{r^{\beta+\gamma-1}}{1+r^\beta\rho^\alpha}dr\right|^2 \rho^{\ell-1} d\rho<\infty.
\end{equation}

The second term on the right-hand side of \eqref{J:two} is more involved.
By \eqref{par}, we have

\begin{equation*}
\frac{d}{dr}\left[r^{\beta+\gamma-1}E_{\beta,\beta+\gamma}\left(-r^\beta\rho^\alpha\right)\right]
=r^{\beta+\gamma-2}E_{\beta,\beta+\gamma-1}\left(-r^\beta\rho^\alpha\right).
\end{equation*}
\textcolor{black}{Then, by integration by parts, we obtain}
\begin{equation*}
	J_2(\rho)=
	-ie^{iR}R^{\beta+\gamma-1}E_{\beta,\beta+\gamma}\left(-R^\beta\rho^\alpha\right)
	+ie^iE_{\beta,\beta+\gamma}\left(-\rho^\alpha\right)
	+i\int_1^R e^{ir}r^{\beta+\gamma-2}E_{\beta, \beta+\gamma-1}\left(-r^\beta\rho^\alpha\right)dr.
\end{equation*}
This, together with \eqref{Upper:E} and the triangle inequality, implies that
\begin{equation*}
	\int_0^\infty|J_2(\rho)|^2\rho^{\ell-1}d\rho\le c_{2,7}(1+D_1(R)+D_2(R)),
\end{equation*}
where
\[
D_1(R):=\int_0^\infty\left(\frac{R^{\beta+\gamma-1}}{1+R^\beta\rho^\alpha}\right)^2\rho^{\ell-1}d\rho,
\]
\[
D_2(R):=\int_0^\infty\left(\int_1^R\frac{r^{\beta+\gamma-2}}{1+r^\beta\rho^\alpha}dr\right)^2\rho^{\ell-1}d\rho,
\]
and  the integral
\begin{equation*}
	\int_0^\infty\frac{\rho^{\ell-1}}{\left(1+\rho^\alpha\right)^2}d\rho<\infty,
\end{equation*}
\textcolor{black}{is finite because $0<\ell<2\alpha$.}\par
\textbf{Step 3.} Estimates for $D_1(R)$ and $D_2(R)$.\\
\textbf{(i)} For $D_1(R)$, using the change of variables $\zeta:=R^{\frac\beta\alpha}\rho$, we have
\begin{equation}\label{D1}
	D_1(R)
	=R^{2(\beta+\gamma-1)-\frac{\beta \ell}\alpha}\int_0^\infty\frac{\zeta^{\ell-1}}{\left(1+\zeta^\alpha\right)^2}d\zeta
	=c_{2,8}R^{2(H-H_0)}. 	 
\end{equation}  
\textbf{(ii)} For $D_2(R)$, using  Minkowski's inequality yields
\begin{equation*}
	D_2(R)\le \left(\int_1^R r^{\beta+\gamma-2}\left(\int_0^\infty\frac{\rho^{\ell-1}}{\left(1+r^\beta\rho^\alpha\right)^2 }d\rho\right)^{\frac12}   dr\right)^2.
\end{equation*}
By the change of variables $q:=r^{\frac\beta\alpha}\rho$ and $0<\ell<2\alpha$, we have
\[
\int_0^\infty\frac{\rho^{\ell-1}}{\left(1+r^\beta\rho^\alpha\right)^2 }d\rho=r^{-\frac{\beta \ell}\alpha}
\int_0^\infty\frac{q^{\ell-1}}{\left(1+q^\alpha\right)^2}dq=c_{2,8}r^{-\frac{\beta \ell}\alpha}.
\]
Therefore, we have 
\begin{equation}\label{D2}
\begin{split}
	D_2(R)\le &\,c_{2,8}\left(\int_1^R r^{\beta+\gamma-2-\frac{\beta \ell}{2\alpha}}dr\right)^2=c_{2,8}\left(\int_1^Rr^{H-H_0-1}dr\right)^2\\
    =&\, c_{2,9}\begin{cases}
		1,&  H< H_0, \\[2ex]
		(\log R)^2,&  H= H_0, \\[2ex]
		R^{2(H-H_0)}, &  H> H_0.	
	\end{cases}
    \end{split}
\end{equation}

It follows from \eqref{J:two}, \eqref{J1}, \eqref{D1} and \eqref{D2} that there exists a constant $c_{2,10}>0$, independent of $R$, such that for all $R> 1$,
\begin{equation}\label{JR2}
	J(R)\le c_{2,10}\begin{cases}
		1, & H< H_0, \\[1.5ex]
		(\log R)^2, & H= H_0, \\[1.5ex]
		R^{2(H-H_0)}, & H> H_0.
	\end{cases}
\end{equation}
Combining \eqref{JR1} and \eqref{JR2}, we get \eqref{U:JR:eq}. The proof is complete.

\end{proof}

\begin{proposition}\label{Harmonizable}
Recall that $H$ is given in \eqref{def:H}. Assume  that Condition \ref{hypo} and $H<1$ hold.	There exists a constant $c_{2,11}>0$ such that for all $1\le a<b\le\infty$, $0<t\leq a^{-\frac1{H}},x\in\RR^d$,
\begin{equation}\label{U:vx:L2}
\begin{split}
	&\|v(\RR_+,t,x)-v([a,b),t,x)\|_{L^2(\Omega)}\\
	\le &\,  c_{2,11}\left(a^{\frac{1-H_0}{H}}t^{H+1-H_0}+b^{-1}\bm{1}_{\{ 0<t<b^{-1/H}\}} +F(b,t,H,H_0)\bm{1}_{\{t\ge b^{-1/H}\}} \right),
    \end{split}
\end{equation}
\textcolor{black}{where, for $b<\infty$ and $t\ge b^{-1/H}$,}
\begin{equation}\label{def:F}
	F(b,t,H,H_0):=\begin{cases}	
		b^{-1}, &  H<H_0,\\[1.5ex]
		b^{-1}\left(\log (b^{1/H}t) +  \left(\log (b^{1/H}t)\right)^{1/2}+1\right), &  H=H_0,\\[1.5ex]
		b^{-\frac{H_0}H}t^{(H-H_0)}, &  H>H_0.
	\end{cases}
\end{equation}

\end{proposition}

\begin{proof}
When $b=\infty$, we use the convention $b^{-1}=0$ and $F(\infty,t,H,H_0)=0$. In what follows, we assume that $b<\infty$.
First,
$$
v(\RR_+,t,x)-v([a,b),t,x)=v([0,a),t,x)+v([b,\infty),t,x).
$$
By \eqref{def:vx}, we have
\begin{equation}\label{ee:two:I1:I2}
	\ee[|v(\RR_+,t,x)-v([a,b),t,x)|^2]\le \frac1{\pi}(I_1+I_2),
\end{equation}
where
\begin{equation*}
	I_1:=\int_{|\tau|< a^{\frac1H}}   |\tau|^{1-2H_0} d\tau   \int_{\RR^d}
	\left|\mathcal FG_{t,x}(\tau,\xi)\right|^2 |\xi|^{\ell-d} d\xi,
\end{equation*}
\begin{equation}\label{def:I2}
	I_2:=\int_{|\tau|\ge b^{\frac1H}}   |\tau|^{1-2H_0} d\tau   \int_{\RR^d}
	\left|\mathcal FG_{t,x}(\tau,\xi)\right|^2 |\xi|^{\ell-d} d\xi.
\end{equation}

By \eqref{Four:G:tx} and the change of variables $r:=|\tau|s$, $\eta:=|\tau|^{-\frac\beta\alpha}\xi$, we have
\begin{equation}\label{equ:Rd}
	\begin{split}
		\int_{\RR^d}
		\left|\mathcal FG_{t,x}(\tau,\xi)\right|^2 |\xi|^{\ell-d} d\xi
		=&\,  \int_{\RR^d}\left| \int_0^te^{i\tau s}s^{\beta+\gamma-1}E_{\beta, \beta+\gamma}\left(-s^\beta|\xi|^\alpha\right)ds\right|^2 |\xi|^{\ell-d} d\xi\\
		=&\, |\tau|^{-2(H+1-H_0)}J(|\tau|t),
	\end{split}
\end{equation}
where $J(\cdot)$ is defined in \eqref{def:J}.\\
\textbf{(1)} For $I_1$:
Since $0<t\leq a^{-\frac1{H}}$ and $0< |\tau|<a^{\frac1H}$, we have $0<|\tau|t<1$. In this case, by Lemma \ref{U:JR} and \eqref{equ:Rd}, we have
\begin{equation}\label{I1}
	\begin{split}
	I_1=&\int_{|\tau|<a^{\frac1H}}|\tau|^{1-2H_0-2(H+1-H_0)}J(|\tau|t)d\tau\\
	\le&\, c_{2,2}t^{2(H+1-H_0)}\int_{|\tau|<a^{\frac1H}}|\tau|^{1-2H_0}d\tau\\
	=&\,\frac{c_{2,2}}{1-H_0}a^{\frac{2-2H_0}H}t^{2(H+1-H_0)}.
\end{split}
\end{equation}
\textbf{(2)} For $I_2$: By \eqref{def:I2} and \eqref{equ:Rd}, we divide $I_2$ into the  following two parts:     
\[
I_2\le I_{2,1}+I_{2,2},
\]
where 
\begin{align*}
I_{2,1}:=&\, \int_{\{|\tau|\ge b^{1/H}, 0<|\tau |\le1/t \}    }  |\tau|^{-1-2H}   J(|\tau|t) d\tau,\\
I_{2,2}:= &\, \int_{ \{|\tau|\ge b^{1/H},  |\tau|>1/t \}    }  |\tau|^{-1-2H}   J(|\tau|t) d\tau.
\end{align*}
\textcolor{black}{We distinguish the following two cases:}
\[
0< b^{1/H}<1/t\quad \text{and}\quad b^{1/H}\ge1/t.
\]
\textbf{Case 1:}
When $0< b^{1/H}<1/t$, by Lemma \ref{U:JR}, we have
\begin{equation*}
	\textcolor{black}{I_{2,1}\le 2c_{2,2}\int_{ b^{1/H}}^{1/t}  \tau^{-1-2H}d\tau
	\le 2c_{2,2} \int_{ b^{\frac1H}}^\infty \tau^{-1-2H}d\tau =\frac{c_{2,2}}{H}b^{-2}.}
\end{equation*}
For $I_{2,2}$,  by Lemma \ref{U:JR}, we have 
\begin{itemize}
	\item [(i)] when $H<H_0$,
	\[
	I_{2,2}\le   2c_{2,2}\int_{1/t}^\infty\tau^{-1-2H}d\tau=\frac{c_{2,2}}{H}t^{2H};
	\]
	\item [(ii)]  when $H=H_0$,
	\begin{align*}
	I_{2,2}\le &\, 2c_{2,2}\int_{1/t}^\infty \tau^{-1-2H}(\log (\tau t))^2d\tau\\
    =&\, 2c_{2,2}t^{2H}\int_1^\infty\eta^{-1-2H}(\log\eta)^2 d\eta=c_{2,12}t^{2H};
	\end{align*}
	\item [(iii)] when $H>H_0$,
	\[
	I_{2,2}\le 2c_{2,2} t^{2(H-H_0)}\int_{1/t}^\infty \tau^{-1-2H+2(H-H_0)}d\tau=\frac{c_{2,2}}{H_0}t^{2H},
	\]	
\end{itemize}
where,  in case (ii), we have used the change of variables $\eta:=\tau t$ and the following integral identity:
\begin{equation*}
	\int_a^\infty x^{-1-b}(\log x)^2dx=a^{-b}\left(\frac{(\log a)^2}{b}+\frac{2\log a}{b^2}+\frac2{b^3}\right),\,\,\,\text{for all}\,\,\,a,b>0.
\end{equation*}
\textcolor{black}{Consequently, there exists a constant $c_{2,15}>0$ such that, whenever $0<b^{1/H}<1/t$,}
\begin{equation}\label{I2:1}
	I_2\le c_{2,13}\left(b^{-2}+t^{2H}\right)\le 2c_{2,13} b^{-2}.
\end{equation}
\textbf{Case 2:} When $b^{1/H}\ge1/t$, $I_{2,1}$ vanishes, so \textcolor{black}{it remains to estimate $I_{2,2}$}.
For $I_{2,2}$,  by  Lemma \ref{U:JR}, we have  
\begin{itemize}
	\item [(i)]  when  $H<H_0$,
	\[
	I_{2,2}\le 2c_{2,2}\int_{ b^{\frac1H}}^\infty   \tau^{-1-2H}   d\tau=\frac{c_{2,4}}{H}b^{-2};
	\]
	\item [(ii)] 	when $H= H_0$,
	\begin{equation*}
		\begin{split}
			I_{2,2}\le&\,  2c_{2,2} \int_{ b^{\frac1H}}^\infty \tau^{-1-2H}(\log \tau t)^2d\tau\\
			=&\,2c_{2,2}b^{-2}\left(\frac{(\log (b^{1/H}t))^2}{2H}+\frac{\log (b^{1/H}t)}{2H^2}+\frac1{4H^3}\right);
		\end{split}
	\end{equation*}
	\item [(iii)]  when $H> H_0$, 
	\begin{equation}\label{I2:2}
		I_{2,2}\le 2c_{2,2}t^{2(H-H_0)}\int_{b^{\frac1H}}^\infty 	 \tau^{-1-2H+2(H-H_0)}  d\tau=\frac{c_{2,2}}{H_0}b^{-\frac{2H_0}H}t^{2(H-H_0)}.
	\end{equation}
\end{itemize}
Consequently, there exists a constant $c_{2,14}>0$ such that, whenever $b^{1/H}\ge1/t$,
\[
\textcolor{black}{I_2\le c_{2,14}F(b,t,H,H_0)^2,}
\]
where $F$ is defined in \eqref{def:F}.

\textcolor{black}{Finally, combining \eqref{I1}, \eqref{I2:1}, and the preceding estimate with \eqref{ee:two:I1:I2} and using $\sqrt{x+y}\le\sqrt{x}+\sqrt{y}$ for $x,y\ge0$, we get \eqref{U:vx:L2}.} The proof is complete.
\end{proof}

\textcolor{black}{From now on, for fixed $x\in\mathbb R^d$, we work with the continuous modification of the harmonizable version $u(t,x):=v(\mathbb R_+,t,x)$. By the preceding moment estimate, this version has continuous paths on every compact time interval. The harmonizable version and the stochastic-convolution version are centered Gaussian random elements of $C([0,T])$ with the same finite-dimensional distributions; hence their induced probability measures on $C([0,T])$ coincide. Therefore, all pathwise LIL statements proved below for the harmonizable version also hold for the original continuous stochastic-convolution solution.}

\section{Proof of Theorem \ref{Khinchin:u}}

For a real-valued, centered Gaussian process $\{X(t)\}_{t \in S}$,
denote  
\[
d_X(s, t) := \|X(t) - X(s)\|_{L^2(\Omega)} = \left(\mathbb{E}\left[|X(t) - X(s)|^2\right] \right)^{1/2}, \quad s, t \in S.
\]
We also denote by $ D_X(S)$ the diameter of $S$ in the metric $d_X$, that is
$$
D_X(S) := \sup\left\{d_X(s, t) : s, t \in S\right\}.
$$
Let $N(S, d_X; \varepsilon)$ denote the smallest number of open $d_X$-balls of radius $\varepsilon$ required to cover $S$.

\begin{lemma} \cite[Theorem 2.4]{T1994}\label{Talagrand}
Let $\{X(t)\}_{t\in S}$ be a mean-zero continuous Gaussian process.
Denote $\sigma^2 := \sup_{t\in S} \|X(t)\|_{L^2(\Omega)}^2$. Suppose that  there exist constants $M > \sigma$, $p > 0$, and $0 < \varepsilon_0 \leqslant \sigma$ such that
\[
N(S, d_X, \varepsilon) \leq (M/\varepsilon)^p \quad \text{for all } \varepsilon < \varepsilon_0.
\]
Then, for any $u > \sigma^2[(1 + \sqrt{p})/\varepsilon_0]$, the following holds:
\[
\mathbb{P}\left\{ \sup_{t\in S} X(t) \geq u \right\} \leq \left(\frac{KMu}{\sqrt{p}\,\sigma^2} \right)^p \Phi\left(\frac{u}{\sigma} \right),
\]
where $\Phi(x) = (2\pi)^{-1/2} \int_x^\infty e^{-y^2/2}\,dy$ and $K$ is a universal constant.
\end{lemma}

The following Gaussian estimate is standard:
\begin{equation}\label{Phi}
\frac{1}{2\sqrt{2\pi}x}e^{-x^2/2} \leqslant \Phi(x) \leqslant \frac{1}{\sqrt{2\pi}}e^{-x^2/2} \quad \text{for all } x \geq 1.
\end{equation}

\begin{proof}[Proof of Theorem \ref{Khinchin:u}]
The proof of Theorem \ref{Khinchin:u} is inspired by Qian et al. \cite[Proposition 5.1(a)]{QWWX2026}.
Fix $x\in\RR^d$. \textcolor{black}{For $0<r<e^{-1}$, define}
$$
L(r):=\sup_{0<t\le r}\frac{|u(t,x)|}{t^{H}\sqrt{2\log\log(1/t)}}.
$$
\textcolor{black}{Since $L(r)\downarrow\limsup_{t\downarrow0}|u(t,x)|/[t^H\sqrt{2\log\log(1/t)}]$ as $r\downarrow0$, it is enough to prove that}
\begin{equation}\label{U:Lr}
	\lim_{r\rightarrow 0^+}L(r)\le \widetilde{\kappa},\,\,\text{a.s.},
\end{equation}
\begin{equation}\label{L:Lr}
	\lim_{r\rightarrow 0^+}L(r)\ge \textcolor{black}{\widetilde{\kappa}},\,\,\text{a.s.}
\end{equation}
We establish the two bounds separately.\\
\textbf{Upper bound.} Let $a>1$ and $\zeta>0$ be constants. For each $n\ge1$, let
\begin{equation*}
	r_n:=a^{-n}\,\,\,\text{and}\,\,\,\theta_n:=(1+\zeta)\textcolor{black}{\widetilde{\kappa}}r_n^{H}\sqrt{2\log\log(1/r_n)}.
\end{equation*}
Consider the event
$$
A_n:=\left\{\sup_{0<t\le r_n}|u(t,x)|> \theta_n\right\}.
$$
\textcolor{black}{We apply Lemma~\ref{Talagrand} to estimate $\mathbb P(A_n)$.} For any $n\ge 1$, by \eqref{u-L2}, we have
$$
\sigma^2:=\sup_{0< t\le r_n}\|u(t,x)\|_{L^2(\Omega)}^2=\textcolor{black}{\widetilde{\kappa}}^2r_n^{2H}.
$$
For every $\varepsilon>0$, let $s_i=i(\varepsilon/c_{2,1})^{1/H}$ for $i=0,1,\ldots,N$, where $N$ is the smallest integer such that $s_N\ge r_n$. Then, by Proposition \ref{solution:holder},  we have
$$
d_u(s_i,s_{i-1})\le \varepsilon.
$$
Then, for any $0<\varepsilon<\sigma$, the covering number $N([0,r_n],d_u,\varepsilon)$ satisfies
\begin{equation}\label{cov:num}
N([0,r_n],d_u,\varepsilon)\le 2r_n\left(\frac{\varepsilon}{c_{2,1}}\right)^{-\frac1H}.
\end{equation}
  {Indeed, $N\le r_n(c_{2,1}/\varepsilon)^{1/H}+1$. Since $\varepsilon<\sigma=\widetilde\kappa r_n^H$ and $c_{2,1}\ge 2^{H+1}\widetilde\kappa$, the first term on the right is larger than $1$, and hence $N\le2r_n(c_{2,1}/\varepsilon)^{1/H}$.}

\textcolor{black}{For all sufficiently large $n$,} $\theta_n>\textcolor{black}{\widetilde{\kappa}}r_n^{H}(1+\sqrt{1/{H}})$. Applying Lemma \ref{Talagrand} with $\varepsilon_0=\sigma$, $p=\frac1{H}$ and  $M=c_{2,1}2^H r_n^H>\sigma$, we have
$$
\mathbb P(A_n)\le 2\left(\frac{K2^{H}c_{2,1}\theta_n} { \sqrt{1/{H}}\textcolor{black}{\widetilde{\kappa}}^2 r_n^{H} } \right)^{\frac1{H}}\Phi\left(\frac{\theta_n}{\textcolor{black}{\widetilde{\kappa}} r_n^{H}}\right),
$$
where $\Phi(x) = (2\pi)^{-1/2} \int_x^\infty e^{-y^2/2}dy$ and $K$ is the universal constant in Lemma \ref{Talagrand}.

By the estimate \eqref{Phi}, there exists a constant \(c_{3,1}>0\) such that,
for all sufficiently large \(n\),
\[
        \mathbb P(A_n)
        \le
        c_{3,1}(\log n)^{\frac1{2H}} n^{-(1+\zeta)^2}.
\]
Since \((1+\zeta)^2>1\), it follows that
\[
        \sum_{n=1}^{\infty}\mathbb P(A_n)<\infty .
\]
Therefore, by the Borel--Cantelli lemma, almost surely,
\[
        \limsup_{n\to\infty}
        \sup_{0<t\le r_n}
        \frac{|u(t,x)|}
        {r_n^{H}\sqrt{2\log\log(1/r_n)}}
        \le
        (1+\zeta)\widetilde\kappa, 
\]
Consequently, for every $r_{n+1}<t\le r_n$, since $r_n/r_{n+1}=a$, we have for all sufficiently large $n$, 
\[
\sup_{r_{n+1}<t\le r_n}
\frac{|u(t,x)|}{t^H\sqrt{2\log\log(1/t)}}
<
\sup_{0<t\le r_n}\frac{|u(t,x)|}{r_{n+1}^H\sqrt{2\log\log(1/r_n)}}
\le a^H(1+\zeta)\widetilde\kappa,\qquad \text{a.s.}
\]
Taking the limsup over $t\downarrow0$ and then letting $a\downarrow1$ and $\zeta\downarrow0$ along rational sequences, we get \eqref{U:Lr}.

\textbf{Lower bound.} Fix $0<\varepsilon<1$. Let $0<\delta<1$ be a small fixed number (depending on $\varepsilon$) to be determined later. For each $n\ge1$, let
\begin{equation}\label{def:tn:rhon}
	t_n:=\rho_n^{\frac1{H}}\,\,\text{with}\,\,\,\,\rho_n:=\exp\left(-\left(n^\delta+n^{1+\delta}\right)\right).
\end{equation}
\textcolor{black}{By Proposition~\ref{Harmonizable}}, we can write $u(t,x)=v_n(t)+\tilde{v}_n(t)$, where
$$
v_n(t):=v([b_n,b_{n+1}),t,x),\,\,\,\tilde{v}_n(t):=v(\RR_+\setminus [b_n,b_{n+1}),t,x),
$$
with $b_n:=\exp\left(n^{1+\delta}\right)$.

For every fixed $x\in\RR^d$, it suffices to prove the following two assertions:
\begin{equation}\label{U:vn}
	\limsup_{n\rightarrow\infty}\frac{|v_n(t_n)|}{t_n^{H}\sqrt{2\log\log(t_n^{-1})}}\ge (1-\varepsilon)\textcolor{black}{\widetilde{\kappa}},\,\,\text{a.s.},
\end{equation}
and
\begin{equation}\label{L:tilde:vn}
	\limsup_{n\rightarrow\infty}\frac{|\tilde{v}_n(t_n)|}{t_n^{H}\sqrt{2\log\log(t_n^{-1})}}\le \varepsilon,\,\,\text{a.s.}
\end{equation}

To prove \eqref{U:vn}, for each $n\ge1$, we define the event
$$
B_n:=\left\{|v_n(t_n)|\ge(1-\varepsilon)\textcolor{black}{\widetilde{\kappa}} t_n^{H} \sqrt{2\log\log(t_n^{-1})}  \right\}.
$$
Let $D_n([0,t_n])$ denote the diameter of the interval $[0,t_n]$ with respect to the canonical metric $$d_{\tilde{v}_n}:=\|\tilde{v}_n(s)-\tilde{v}_n(t)\|_{L^2(\Omega)},$$ where $t_n$ is given by \eqref{def:tn:rhon}.
Then, by Proposition \ref{Harmonizable}, we have
\begin{equation*}
	\sup_{t\in[0,t_n]}\|\tilde{v}_n(t)\|_{L^2(\Omega)}
	\leq c_{2,11}\left({b_n}^{\frac{1-H_0}{H}}{t_n}^{H+1-H_0}+{b_{n+1}}^{-1} + F(b_{n+1},t_n, H, H_0)\right),
\end{equation*}
where $F$ is given by \eqref{def:F}. 

Note that  $b_n\rho_n=\exp\left(-n^\delta\right)$. Also, by the mean value theorem,
$$
(n+1)^{1+\delta}-n^{1+\delta}\ge (1+\delta)n^\delta,
$$
which implies
$b_{n+1}\rho_n\ge\exp\left(\delta n^\delta\right)$. We have the following estimates:
\begin{itemize} 
	\item [(i)]
	\begin{equation*}
		{b_n}^{\frac{1-H_0}{H}}{t_n}^{H+1-H_0}=\rho_n(b_n\rho_n)^{\frac{1-H_0}{H}}=\rho_n\exp\left(- \frac{1-H_0}{H}n^\delta  \right),
	\end{equation*}
	\item [(ii)]
	\begin{equation*}
		{b_{n+1}}^{-1}=\rho_n(b_{n+1}\rho_n)^{-1}\le \rho_n\exp\left(-\delta n^\delta\right),
	\end{equation*}
	\item [(iii)] Since 
	\[b_{n+1}^{1/H}t_n=(b_{n+1}\rho_n)^{1/H}\ge\exp\left(\frac\delta H n^\delta\right)\ge 1,
	\]
	there exist  constants $0<\theta<H$ and $c_\theta>0$, depending on $\theta$, such that  when $H=H_0$,
	\begin{equation*}
		\begin{split}
			&\, {b_{n+1}}^{-1}\left(\log \left(b_{n+1}^{1/H}t_n\right) + \left(\log \left(b_{n+1}^{1/H}t_n\right)\right)^{1/2}+1\right)\\
			\le&\, c_\theta {b_{n+1}}^{-1+\frac\theta{H}}{t_n}^{\theta}
			= c_\theta\rho_n(b_{n+1}\rho_n)^{-\left(1-\frac\theta{H}\right)}
			\le c_\theta \rho_n\exp\left(-\left(1-\frac\theta{H}\right)  \delta n^\delta\right),\,\,\, 
		\end{split}	
	\end{equation*}
	and when $H>H_0$,
	\begin{equation*}
		{b_{n+1}}^{-\frac{H_0}H}{t_n}^{H-H_0}=\rho_n(b_{n+1}\rho_n)^{-\frac{H_0}H}\le \rho_n\exp\left(-\frac{H_0}H\delta n^\delta\right).
	\end{equation*}
\end{itemize}
Combining (i)-(iii) above, we see that there exists a constant $c^*>0$ such that
\begin{equation}\label{U:Dn}
	D_n([0,t_n])\le 2\sup_{t\in[0,t_n]}\|\tilde{v}_n(t)\|_{L^2(\Omega)}\le c_{3,2}\rho_n\exp\left(-c^* n^\delta\right).
\end{equation}
By the triangle inequality, \eqref{u-L2} and \eqref{U:Dn}, we have
\begin{equation}\label{L:vn}
\begin{split}
	\|v_n(t_n)\|_{L^2(\Omega)}\ge  &\,  \|u(t_n,x)\|_{L^2(\Omega)}-\|\tilde{v}_n(t_n)\|_{L^2(\Omega)}\\
	\ge &\, \left (\textcolor{black}{\widetilde{\kappa}}-c_{3,2}\exp\left(-c^* n^\delta\right)\right )t_n^{H}.
    \end{split}
\end{equation}
\textcolor{black}{Moreover, $\log\log(t_n^{-1})/\log\log(\rho_n^{-1})\to1$. Together with \eqref{L:vn}, this implies that, for all sufficiently large $n$,}
$$
B_n\supset \left\{|v_n(t_n)|\ge(1-\varepsilon/2)\|v_n(t_n)\|_{L^2(\Omega)}\sqrt{2\log\log(1/\rho_n)} \right\}.
$$
By the standard Gaussian estimate \eqref{Phi}, for all sufficiently large $n$,
$$
\mathbb P(B_n)\ge c_{3,3}(\log n)^{-1/2} n^{-(1-\varepsilon/2)^2(1+\delta)},
$$
where $c_{3,3}\in(0,\infty)$. Choosing $\delta$ sufficiently small so that $(1-\varepsilon/2)^2(1+\delta)\le 1$, it follows that
$$
\sum_{n=1}^\infty\mathbb P(B_n)=\infty.
$$
\textcolor{black}{Since the frequency blocks $[b_n,b_{n+1})$ are pairwise disjoint, the processes $\{v_n\}_{n\ge1}$ are independent. Hence, by the second Borel--Cantelli lemma, we obtain \eqref{U:vn}.}

For \eqref{L:tilde:vn}, we use \eqref{U:Dn} to get that
\begin{equation*}
	\begin{split}
		&\, \mathbb P\left(|\tilde{v}_n(t_n)|\ge \varepsilon
		t_n^{H}\sqrt{2\log\log(t_n^{-1})}\right)\\
		\le&\, \mathbb P\left(|\tilde{v}_n(t_n)|\ge\frac \varepsilon{c_{3,2}}\|\tilde{v}_n(t_n)\|_{L^2(\Omega)}\exp\left(c^* n^\delta\right)
		\sqrt{2\log\log(\rho_n^{-1})}\right).
	\end{split}
\end{equation*}
\textcolor{black}{By the standard two-sided Gaussian tail bound, this probability is at most}
$$
\textcolor{black}{2\exp\left(-\frac{\varepsilon^2}{2c_{3,2}^2}\exp\left(2c^* n^\delta\right)
\log\log(1/\rho_n)\right)\le c_{3,4}n^{-2},}
$$
for all sufficiently large $n$, where $c_{3,4}\in(0,\infty)$. Therefore, the Borel--Cantelli lemma yields \eqref{L:tilde:vn}.

Since $u(t,x)=v_n(t)+\tilde{v}_n(t)$, combining \eqref{U:vn} and \eqref{L:tilde:vn} yields
$$
\limsup_{n\rightarrow\infty}\frac{|u(t_n,x)|}{t_n^{H}\sqrt{2\log\log(t_n^{-1})}}\ge (1-\varepsilon)\textcolor{black}{\widetilde{\kappa}}-\varepsilon,\,\,\text{a.s.}
$$
Letting $\varepsilon\downarrow0$ along a rational sequence, we get \eqref{L:Lr}. The proof is complete.
\end{proof}

\section{Proof of Theorem \ref{chung:u}}

\subsection{  Small-ball estimate for the solution}
We next state a small-ball estimate for the process $\{u(t, x)\}_{t\ge0}$ at a spatial point  $x\in \mathbb R^d$. This estimate will play  a key role in the sequel.  
\begin{proposition}\label{small-ball:u} 
	 Suppose that Condition \ref{hypo} holds, $H<1$,   {$0\le\gamma<1$}, and $\beta+\gamma<2+H$.
Then, for any  fixed $x\in\RR^d$,
\begin{equation}\label{pr:u}
	\lim\limits_{\varepsilon\downarrow0}\varepsilon^{\frac1H}\log
	\bp\left(\sup_{t\in[0,1]}|u(t,x)|\le\varepsilon\right)
	=-\kappa^{\frac1{H}}\lambda_H,
\end{equation}
where $\kappa$ and $\lambda_H$ are given by \eqref{def:kappa} and  \eqref{def:lambda:H}, respectively.
\end{proposition}

\begin{proof}
  {The exact small-ball result \cite[Theorem~4.3(1)]{GSWX2025} is formulated for $0\le\gamma<1$, which explains this hypothesis. We next verify the Fourier normalization and the correspondence between its parameters and those used here.} Under the Fourier-domain covariance convention \eqref{e:covariance-noise-2}, the temporal Fourier factor is $(2\pi)^{-1}$. Therefore, the stationary-increment component used in Theorem~4.3(1) of \cite{GSWX2025} has, at a fixed spatial point, temporal spectral density
\[
\textcolor{black}{
\frac{1}{2\pi}|\tau|^{1-2\gamma-2H_0}
\int_{\mathbb R^d}
\frac{|\xi|^{\ell-d}}
{|\tau|^{2\beta}+2|\tau|^\beta|\xi|^\alpha\cos(\pi\beta/2)+|\xi|^{2\alpha}}\,d\xi .}
\]
\textcolor{black}{With the change of variables $\xi=|\tau|^{\beta/\alpha}\eta$, the spatial integral contributes the factor $|\tau|^{\beta\ell/\alpha-2\beta}$. By \eqref{def:H}, the resulting temporal density equals}
\[
\textcolor{black}{
\frac{I_{\alpha,\beta,\ell}}{2\pi}|\tau|^{-2H-1},
\qquad
I_{\alpha,\beta,\ell}:=
\int_{\mathbb R^d}
\frac{|\eta|^{\ell-d}}
{1+2|\eta|^\alpha\cos(\pi\beta/2)+|\eta|^{2\alpha}}\,d\eta .}
\]
\textcolor{black}{Consequently, the increment variance of this component is}
\[
\textcolor{black}{
\frac{I_{\alpha,\beta,\ell}}{\pi}
\int_{\mathbb R}\frac{1-\cos((t-s)\tau)}{|\tau|^{2H+1}}\,d\tau
=\kappa^2|t-s|^{2H},}
\]
where $\kappa$ is the constant in \eqref{def:kappa}.
Thus the stationary-increment component has the same
finite-dimensional distributions as $\kappa B_H$. 

\textcolor{black}{It remains to control the smoother remainder. The proofs of Proposition~4.2(1) and Theorem~4.3(1) in \cite{GSWX2025} use only the homogeneity of the spatial spectral measure, the Mittag--Leffler estimate \eqref{Upper:E}, and the differentiability identity \eqref{par}. The same estimates therefore apply to the Riesz-type density $|\xi|^{\ell-d}$. The additional condition required for the remainder is automatic when $\beta+\gamma\le2$. When $\beta+\gamma>2$, it becomes}
\[
\textcolor{black}{-H_0+\frac{\beta\ell}{2\alpha}<1,}
\]
\textcolor{black}{which is equivalent, by \eqref{def:H}, to $\beta+\gamma<2+H$. Since $H<1$, the remainder has a strictly smaller small-deviation exponent than the fractional Brownian component. The Gaussian correlation argument in the proof of \cite[Theorem~4.3(1)]{GSWX2025} therefore yields the same exact small-ball constant as that of $\kappa B_H$, which proves \eqref{pr:u}.}
\end{proof}

\textcolor{black}{We also use the following Gaussian isoperimetric inequality.}
\begin{lemma}\label{isoperimetric}
There  is a universal constant $K_0$ such that the following statement holds. Let $S$ be a bounded set and $\{X(s),s\in S\}$ be a separable Gaussian process.
Let $D_X(S) := \sup\left\{d_X(s, t):s, t \in S\right\}$ be the diameter of $S$ in the metric $d_X$. Then for any $u>0$,
\[
\mathbb P\left(\sup_{s,t\in S}   |X(s)-X(t)|\ge K_0\left(u+\int_0^{D_X(S)} \sqrt{\log N(S,d_X;\varepsilon)}  d\varepsilon    \right)  \right)\le \exp\left(-\frac{u^2}{D_X(S)^2}\right).
\]
\end{lemma}

The following lemma is a standard consequence of Anderson's inequality.
\begin{lemma}{\rm (\cite[Lemma 2.8]{FFM2026})} \label{Gaussian:sym}
Let $X$ and $Y$ be independent centered Gaussian processes on $S$. Then for any $a>0$,
\[
\bp\left(\sup_{t\in S}  |X(t)+Y(t)|\le a\right)\le \bp\left(\sup_{t\in S}  |X(t)|\le a\right).
\]
\end{lemma}

\subsection{Proof of Theorem \ref{chung:u}}
\textcolor{black}{We are now ready to prove Chung's LIL.}
\begin{proof}[Proof of Theorem \ref{chung:u}]
\textcolor{black}{For $0<r<e^{-1}$, set}
\begin{equation}\label{def:psi}
\textcolor{black}{\psi(r):=r^H\bigl(\log\log r^{-1}\bigr)^{-H}
=\left(\frac{r}{\log\log(1/r)}\right)^H.}
\end{equation}
\noindent\textcolor{black}{(1) {\bf Lower bound.}}   In this step, we prove the lower bound:
\begin{equation}\label{u:lower}
	\liminf_{\varepsilon\downarrow0}\sup_{t\in[0,\varepsilon]}
	\frac{|u(t,x)|}{\e^{H}(\log\log \e^{-1})^{-H}}\geq
	\kappa\lambda_H^{ H },\ \ \text{a.s.}
\end{equation}

Fix $c>0$, and define     $r_n:=e^{-cn}$ for $n\ge1$.   For any fixed $x\in \mathbb R^d$ and $0<q<\lambda_H$, define
$$
A_n:=\left\{ \sup_{t\in[0,r_n]}|u(t,x)|    \leq \kappa\left(\frac{ q  r_n}
{  \log \log (1/r_n) }\right)^{H} \right\}.
$$
By the self-similarity of $\{u(t, x)\}_{t\ge0}$ and \eqref{pr:u}, there exists a constant
$N_0>0$ such that for any $n\ge N_0$,
$$
\mathbb P\left(A_n\right) \le (nc)^{ -  \frac{\lambda_H+q}{2q}}.
$$
Consequently,
$$
\sum_{n\in \mathbb N} \mathbb P\left(A_n\right) <\infty.
$$
\textcolor{black}{For each fixed $q\in(0,\lambda_H)$,  Borel--Cantelli's lemma gives}
\[
\textcolor{black}{
\liminf_{n\to\infty}
\frac{\sup_{0\le t\le r_n}|u(t,x)|}{\psi(r_n)}
\ge \kappa q^H,
\qquad\text{a.s.}}
\]
\textcolor{black}{Choosing a countable sequence $q_m\uparrow\lambda_H$ and intersecting the corresponding probability-one events, we obtain}
\[
\textcolor{black}{
\liminf_{n\to\infty}
\frac{\sup_{0\le t\le r_n}|u(t,x)|}{\psi(r_n)}
\ge \kappa\lambda_H^H,
\qquad\text{a.s.}}
\]
The function $\psi$ is increasing on $(0,r_0)$ for some sufficiently
small $r_0>0$. Thus, for $r_{n+1}<r\le r_n$, the numerator is at least its value at $r_{n+1}$, whereas $\psi(r)\le\psi(r_n)$. Hence
\[
\textcolor{black}{
\frac{\sup_{0\le t\le r}|u(t,x)|}{\psi(r)}
\ge
\frac{\sup_{0\le t\le r_{n+1}}|u(t,x)|}{\psi(r_n)}
=
\frac{\psi(r_{n+1})}{\psi(r_n)}
\frac{\sup_{0\le t\le r_{n+1}}|u(t,x)|}{\psi(r_{n+1})}.}
\]
\textcolor{black}{Since $r_{n+1}/r_n=e^{-c}$ and}
\[
\textcolor{black}{
\frac{\log\log r_n^{-1}}{\log\log r_{n+1}^{-1}}
=
\frac{\log(cn)}{\log(c(n+1))}\longrightarrow1,}
\]
\textcolor{black}{we have $\psi(r_{n+1})/\psi(r_n)\to e^{-cH}$. Therefore,}
\[
\textcolor{black}{
\liminf_{r\downarrow0}
\frac{\sup_{0\le t\le r}|u(t,x)|}{\psi(r)}
\ge
\kappa e^{-cH}\lambda_H^H,
\qquad\text{a.s.}}
\]
\textcolor{black}{Taking a countable sequence $c_m\downarrow0$ and intersecting the corresponding probability-one events, we obtain the lower bound \eqref{u:lower}.}

\noindent\textcolor{black}{(2) {\bf Upper bound.}}  In this step, we prove the upper bound:
\begin{equation}\label{u:upper}
	\liminf_{r\rightarrow0}\sup_{t\in[0,r]}\frac{|u(t,x)|}  {{r^{H}(\log\log r^{-1})^{-H}} } \le
	\kappa\lambda_H^{ H},\ \ \text{a.s.}
\end{equation}
The proof of \eqref{u:upper} is inspired by Lee and Xiao \cite[pp. 531-533]{LX2023}.

\textcolor{black}{We continue to use the normalizing function $\psi$ defined in \eqref{def:psi}.}
Fix a sufficiently small $\delta>0$. For any $n\ge 1$, let
\begin{equation*}
	t_n:=\rho_n^{\frac1{H}}\,\,\,\text{with}\,\,\,\,\rho_n:=\exp\left(-\left(n^\delta+n^{1+\delta}\right)\right).
\end{equation*}
\textcolor{black}{By Proposition~\ref{Harmonizable}}, we can write $u(t,x)=v_n(t)+\tilde{v}_n(t)$, where
$$
v_n(t):=v([b_n,b_{n+1}),t,x),\,\,\,\tilde{v}_n(t):=v(\RR_+\setminus [b_n,b_{n+1}),t,x),
$$
with $b_n:=\exp\left(n^{1+\delta}\right)$.
 
For every fixed $x\in\RR^d$, we aim to prove that
\begin{equation}\label{U:vn:Chung}
	\liminf_{n\rightarrow\infty} \sup_{t\in[0,t_n]} \frac{|v_n(t)|}{\psi(t_n)}\le\left((1+2\delta) \kappa^{1/H} \lambda_H  \right)^H,\,\,\text{a.s.},
\end{equation}
and
\begin{equation}\label{U:tilde:vn:Chung}
	\limsup_{n\rightarrow\infty}\sup_{t\in[0,t_n]}  \frac{|\tilde{v}_n(t)|}{\psi(t_n)}=0,\,\,\text{a.s.}
\end{equation}
To prove \eqref{U:vn:Chung},
since $u(t,x)=v_n(t)+\tilde{v}_n(t)$ and the processes $v_n$ and $\tilde{v}_n$ are independent, we can apply
Lemma \ref{Gaussian:sym},  the self-similarity of $\{u(t, x)\}_{t\ge0}$ and  \eqref{pr:u} to obtain that 
there exists a constant $0<\eta<\delta\kappa^{1/H} \lambda_H/(1+\delta)$ such that for all  sufficiently large $n$,
\begin{equation*}
	\begin{split}
		& \pp\left(\sup_{t\in[0,t_n]}  |v_n(t)| \le \left((1+2\delta) \kappa^{1/H} \lambda_H  \right)^H  \psi(t_n)  \right)\\
        \ge &\,
		\pp\left(\sup_{t\in[0,t_n]}  |u(t,x)| \le \left((1+2\delta) \kappa^{1/H} \lambda_H  \right)^H \psi(t_n)  \right)\\
		=&\, \pp\left(\sup_{t\in[0,1]}  |u(t,x)| \le \left((1+2\delta) \kappa^{1/H} \lambda_H  \right)^H  (\log\log t_n^{-1})^{-H}  \right)\\
	\ge&\, \exp\left(  -\frac1{(1+2\delta)}\left(1+\frac{\eta}{\kappa^{1/H} \lambda_H}   \right)     (\log\log t_n^{-1})  \right)\\
    =&\, \left(\frac1H\left(n^\delta+n^{1+\delta}\right)\right)^{-\frac1{(1+2\delta)} 
  \left(1+ \frac{\eta}{\kappa^{1/H} \lambda_H}      \right)}\\
    \ge&\, c_{4,1}n^{-\frac{1+\delta}{1+2\delta}\left(1+\frac{\eta}{\kappa^{1/H} \lambda_H} \right)  },   
 \end{split}
\end{equation*}       
    where $c_{4,1}\in(0,\infty)$. The choice of $\eta$ yields 
  \[
\frac{1+\delta}{1+2\delta}\left(1+\frac{\eta}{\kappa^{1/H} \lambda_H} \right)<1,
  \]  
  and then
\[
\sum_{n=1}^\infty
\mathbb P\left\{
        \sup_{t\in[0,t_n]} |v_n(t)|
        \le
        \left((1+2\delta)\kappa^{1/H}\lambda_H\right)^H
        \psi(t_n)
\right\}
=
\infty .
\]
\textcolor{black}{Because the corresponding frequency blocks are pairwise disjoint, the processes $\{v_n\}_{n\ge1}$ are independent. Hence, by the second Borel--Cantelli lemma, the event inside the probability occurs infinitely often almost surely. This proves \eqref{U:vn:Chung}.}

\textcolor{black}{For \eqref{U:tilde:vn:Chung}, independence implies that variances add. Hence, by Proposition~\ref{solution:holder}, for all $s,t\in[0,t_n]$,}
\[
\|\tilde{v}_n(t)-\tilde{v}_n(s)\|_{L^2(\Omega)}\le \|u(t,x)-u(s,x)\|_{L^2(\Omega)}\le c_{2,1}|t-s|^H.
\]
\textcolor{black}{As in the proof of \eqref{cov:num}, the covering number $N([0,t_n],d_{\tilde v_n},\varepsilon)$ obeys}
\[
N([0,t_n], d_{\tilde{v}_n}, \varepsilon)\le c_{4,2} t_n \varepsilon^{-\frac1H}.
\]
Using \eqref{U:Dn} again, the diameter of $[0,t_n]$ 
under the canonical metric of $\{\tilde{v}_n(t)\}_{t\ge 0}$ is
\begin{equation*}
	D_n([0,t_n])\le c_{3,2}\rho_n\exp\left(-c^* n^\delta\right).
\end{equation*}

Then, by the change of variables $w:=c_{4,2} ^{-1}t_n^{-1}\varepsilon^{\frac1H}$, there exists a constant $c_{4,3}>0$ such that for large $n$,
\begin{equation*}
	\begin{split}
		\int_0^{D_n([0,t_n])}\sqrt{\log N\left([0,t_n],d_{\tilde{v}_n},\varepsilon  \right)}  d\varepsilon
		\le&\, \int_0^{c_{3,2} \rho_n\exp\left(-c^* n^\delta\right)}\sqrt{\log \left(c_{4,2}  t_n \varepsilon^{-\frac1H} \right) } d\varepsilon \\
		=&\,Hc_{4,2}^H  t_n^H\int_0^{ c_{4,2} ^{-1}c_{3,2} ^{\frac1H} \exp\left(-\frac{c^*} H   n^\delta\right)}w^{H-1}\sqrt{-\log w}\,dw\\
		\le &\, c_{4,3} t_n^H\exp\left(-c^* n^\delta\right)n^{\frac\delta 2}.
	\end{split}
\end{equation*}
\textcolor{black}{In the last step, we use} the following elementary estimate: for every
\(q<1\) and \(\eta\in(0,1)\), there exists a constant
\(C=C(q,\eta)>0\) such that, for all \(A\in(0,\eta)\),
\[
        \int_0^A x^{-q}\sqrt{-\log x}\,dx
        \le
        C A^{1-q}\sqrt{-\log A}.
\]
See \cite[Lemma~B.3]{CT2025}.

For every $\zeta>0$, by \eqref{def:psi},  there exists an integer $N_1(\zeta)>0$, depending on $\zeta$, such that for all $n\geq N_1(\zeta)$,
\begin{equation*}
	\int_0^{D_n([0,t_n])}\sqrt{\log N\left([0,t_n],d_{\tilde{v}_n},\varepsilon  \right)}  d\varepsilon 	\leq\zeta \psi(t_n).
\end{equation*}

Since $\widetilde v_n(0)=0$, we have
\[
\sup_{0\le t\le t_n}|\tilde v_n(t)|
\le
\sup_{s,t\in[0,t_n]}|\tilde v_n(t)-\textcolor{black}{\tilde v_n(s)}|.
\]
  Applying Lemma  \ref{isoperimetric} to the increment process on $[0,t_n]$ with $u=\zeta \psi(t_n)$, we obtain that for all sufficiently large $n$,
\begin{equation*}
	\begin{split}
		&\, \pp\left(\sup_{t\in[0,t_n]}|\tilde{v}_n(t)| \ge 2K_0\zeta \psi(t_n)  \right)\\
		\le &\, \pp\left(\sup_{s,t\in [0,t_n]} |\tilde{v}_n(t)-\tilde{v}_n(s)|\ge K_0\left(\zeta \psi(t_n)+\int_0^{D_n([0,t_n])} \sqrt{\log N\left([0,t_n],d_{\tilde{v}_n},\varepsilon\right)}  d\varepsilon    \right) \right)\\
		\le&\, \exp\left(-c_{4,4}\frac {\exp(2c^*n^\delta)} {   (\log((n^\delta+n^{1+\delta})/H) )^{2H} }    \right),
	\end{split}
\end{equation*}
where $K_0$ is the universal constant in Lemma \ref{isoperimetric}.
Hence,
\begin{equation*}
	\sum_{n=1}^\infty\pp\left(\sup_{t\in[0,t_n]}|\tilde{v}_n(t)| \ge 2K_0\zeta \psi(t_n)  \right)<\infty.
\end{equation*}
By the Borel--Cantelli lemma,
\begin{equation*}
	\limsup_{n\rightarrow\infty}\sup_{t\in[0,t_n]}\frac{|\tilde{v}_n(t)|}  {\psi(t_n)}\le2K_0\zeta,\,\,\text{a.s.}
\end{equation*}
Since $\zeta>0$ is arbitrary, we get \eqref{U:tilde:vn:Chung}.

By the triangle inequality and the decomposition
$u=v_n+\widetilde v_n$,
\begin{align*}
\textcolor{black}{
\liminf_{r\downarrow0}
\frac{\sup_{0\le t\le r}|u(t,x)|}{\psi(r)}}
&\textcolor{black}{\le
\liminf_{n\to\infty}
\frac{\sup_{0\le t\le t_n}|u(t,x)|}{\psi(t_n)}}\\
&\textcolor{black}{\le
\liminf_{n\to\infty}
\frac{\sup_{0\le t\le t_n}|v_n(t)|}{\psi(t_n)}
+
\limsup_{n\to\infty}
\frac{\sup_{0\le t\le t_n}|\tilde v_n(t)|}{\psi(t_n)}}\\
&\textcolor{black}{\le
\left((1+2\delta)\kappa^{1/H}\lambda_H\right)^H,
\qquad\text{a.s.}}
\end{align*}
\textcolor{black}{Thus,}
\begin{equation}\label{u:upper:1}
\textcolor{black}{
\liminf_{r\downarrow0}
\frac{\sup_{0\le t\le r}|u(t,x)|}{r^H(\log\log r^{-1})^{-H}}
\le
\left((1+2\delta)\kappa^{1/H}\lambda_H\right)^H,
\qquad\text{a.s.}}
\end{equation}
\textcolor{black}{Taking $\delta$ to zero along a countable sequence and combining \eqref{u:lower} with \eqref{u:upper:1} completes the proof.}
\end{proof}

 \noindent{\bf Funding}:  Not applicable.

 \noindent{\bf Author Contributions}:     C. Liu and R. Wang  contributed to writing, reviewing and editing.

 \noindent{\bf Data Availability Statement}:  No data were used for the research described in the article.

 \noindent {\bf Conflict of Interest}:   The authors declare that they have no conflict of interest.

   \noindent {\bf Clinical trial number}:   Not applicable.

\end{document}